\documentclass[11pt,twoside]{article}

\usepackage[ansinew]{inputenc} 
\usepackage{amsmath,amsfonts,amssymb,amsthm}
\usepackage{pstricks,pst-node,pst-coil,pst-plot,pstricks-add}
\usepackage{graphics,geometry,epsfig}
\usepackage{bbm}

\newcommand{\field}[1]{\mathbb{#1}}
\newcommand{\N}{\field{N}}

\newcommand{\Q}{\field{Q}}

\newcommand{\eps}{\varepsilon}

\newcommand{\norm}[1]{\mbox{$\left\| #1 \right\|$}}           

\newtheorem{theorem}{Theorem}[section]
\newtheorem{lemma}[theorem]{Lemma}

\newtheorem{prop}[theorem]{Proposition}
\newtheorem{assume}{Assumption}
\theoremstyle{plain}

\title{Kato's Theorem on the Integration of Non-Autonomous Linear Evolution Equations}

\author{Jochen Schmid and Marcel Griesemer\\  
\small Fachbereich Mathematik, Universit\"at Stuttgart, D-70569 Stuttgart, Germany\\
\small jochen.schmid@mathematik.uni-stuttgart.de}  

\date{}

\font\notefont=cmsl8 
\begin{document}
\maketitle

\begin{abstract}
This  paper is devoted to a comparison of early works of Kato and
Yosida on the integration of non-autonomous linear evolution equations $\dot{x} = A(t)x$ in Banach
space, where the domain $D$ of $A(t)$ is independent of $t$. Our focus is on the regularity assumed of $t\mapsto A(t)$ and our main objective is to clarify the meaning of the rather
involved set of assumptions given in Yosida's classic and highly influential \emph{Functional
  Analysis}.  We prove Yosida's assumptions to be equivalent to Kato's condition that $t\mapsto A(t)x$ is continuously differentiable for each $x\in D$. 
\end{abstract}

\section{Introduction}
This paper is devoted to a comparison of early works of Kato and
Yosida on the integration of non-autonomous, linear evolution equations in Banach space. Explicitly, we consider the abstract initial value problem  
\begin{equation}\label{ivp}
    \dot{x} = A(t)x,\qquad x(s)=y,
\end{equation}
in the Banach space $X$, where $A(t):D\subset X\to X$ for each $t\in
[0,1]$ is a closed linear operator with a dense domain $D$. 
The initial value $y$ belongs to $D$ and $0\leq s<1$. 
The importance of this problem is based on the vast range of applications and on the fact that 
problems of this kind are still the subject of research. Kato in 1953
assumes that $D$ is independent of $t$ and that $A(t)$  for each $t$ is the generator of a contraction
semigroup \cite{Kato53}. In addition, there are some regularity
assumptions on $t\mapsto A(t)$, which are now understood to be
equivalent to the simple condition that for every $x\in D$ 
\begin{equation}\label{C1-A}
 t\mapsto A(t)x\quad\text{is continuously differentiable}
\end{equation}
in the norm of $X$. These conditions are sufficient for the existence
of a unique evolution system (propagator) $U(t,s)$ such that $t\mapsto
U(t,s)y$ for $y\in D$ is a continuously differentiable solution of the initial value problem
\eqref{ivp} \cite{Kato53,EngelNagel,Pazy}. 

In 1956 and 1970, Kato generalized his above-mentioned result to
time-dependent domains and to linear operators $A(t)$ generating
semigroups that are not necessarily contractive \cite{Kato56,Kato70}.
The $C^1$-condition \eqref{C1-A} first appeared explicitly in
\cite{Kato56}. Meanwhile, Yosida, in the second edition of his classic
and influential \emph{Functional Analysis}, had given a simplified
presentation of Kato's work of 1953 with hypotheses that were adjusted accordingly \cite{YosidaFA2}. Yosida's regularity conditions appear weaker than Kato's $C^1$-condition as they involve no derivative of 
$A(t)$. Yet they are far more complicated. Yosida's account of Kato's
theorem remained unchanged over the last five editions of his book and it has been adopted 
by Reed and Simon, and by Blank, Exner and Havlicek \cite{ReedSimon,BEH}. Due to the authority and tremendous popularity 
of the books by Yosida and by Reed and Simon, Yosida's version of Kato's
theorem in a large scientific community is better known than the refined version 
of Kato stated above. 

In this paper we prove that Yosida's conditions, the above mentioned $C^1$-condition, and Kato's original conditions introduced in 1953 are all strictly equivalent.
Likewise, Yosida's regularity conditions in the case of locally convex
spaces can be simplified \cite{Yosida1965,Schmid2012}.
The equivalence of Kato's 1953-condition and the $C^1$-condition is
fairly easy to prove; it was known to Kato and it is known to the experts
on evolution equations \cite{EngelNagel, Pazy}. It is not entirely obvious, however, and we shall provide a proof
for the reader's convenience. The equivalence of
Yosida's conditions and the $C^1$-condition was discovered by one of
us in connection with the adiabatic theorem \cite{Schmid2011}.  This
equivalence is surprising, in view of the complexity of Yosida's conditions.
Nevertheless the proof is fairly short and the idea is simple: Yosida's assumptions require the uniform
convergence of certain left-sided difference quotients of the map $t\mapsto A(t)x$, $x\in D$. It is not hard to see that this requirement 
implies continuous differentiability, and this is the core of our
proof, see Lemma~\ref{lemma}, below. That the converse holds was known
previously to the experts and is straightforward to prove.

As far as we know the literature, before our work 
a direct comparison of the conditions of Kato and Yosida has never been undertaken. Such a comparison is mentioned neither in 
the monographs \cite{Krein-book,Tanabe79,Pazy,Goldstein,EngelNagel}
nor in the review articles \cite{Kato93,Sch2002}. Of course, Kato's
theorem has been generalized in various directions and for that the
reader is referred to Pazy's book \cite{Pazy} and to Kato's \emph{Fermi lectures}
from 1985 \cite{Kato1985}.

\section{Equivalence of regularity assumptions}

We recall from the introduction that  $A(t):D\subset X\to X$ for each $t\in
[0,1]$ is a closed linear operator with a dense, $t$-independent domain
$D$. We are interested in the case where $A(t)$, for each $t$ is the generator of a strongly continuous contraction semigroup,
but this is not needed for the comparison of regularity assumptions. The bounded invertibility of $1-A(t)$ and $A(t)$, respectively, suffices
to state the assumptions and to prove our theorems. We often write $I$ for the interval $[0,1]$.

Kato in Theorem~4 of \cite{Kato53} made the following assumption:

\begin{assume}[Kato 53] \label{ass:Kato}\text{}
\begin{itemize}
\item[(i)] $B(t,s) =(1-A(t))(1-A(s))^{-1}$ is uniformly bounded on $I\times I$.
\item[(ii)] $B(t,s)$ is of bounded variation in $t$ in the sense that there is an $N\geq 0$ such that
$$
        \sum_{j=0}^{n-1}\|B(t_{j+1},s) - B(t_j,s)\| \leq N<\infty
$$
for every partition $0=t_0<t_1<\cdots <t_n=1$ of $I$, at least for some $s$.
\item[(iii)] $B(t,s)$ is weakly differentiable in $t$ and
$\partial_{t}B(t,s)$ is strongly continuous in $t$, at least for some $s\in I$ 
\end{itemize}
\end{assume} 

Note that the statements (ii) and (iii) hold for all $s\in I$, if they are
satisfied for some $s$. This follows from $B(t,s)=B(t,s_0)B(s_0,s)$.
In the  proof of the Proposition~\ref{thm:Kato}, below, we will see that
conditions (i) and (ii) follow from condition (iii), and that
(iii) is equivalent to the $C^1$-condition \eqref{C1-A}. In 1953, Kato did not seem to be
aware of that but from remarks in \cite{Kato53,Kato56} it becomes
clear that he knew it by 1956. See also Remark~6.2 of \cite{Kato70}
which states that the new result -- Theorem 6.1 of \cite{Kato70} -- reduces
to Theorem~4 of \cite{Kato53} in the situation considered there. 

\begin{prop}[Kato]\label{thm:Kato}
Suppose that for each $t\in I$ the linear operator $A(t):D\subset X\to X$ is closed and that 
$1-A(t)$ has a bounded inverse. Then Assumption~\ref{ass:Kato} is satisfied if
and only if the $C^1$-condition \eqref{C1-A} holds.
\end{prop}

\begin{proof}
>From (iii) its follows (first in the weak, then in the strong sense) that 
\begin{equation}\label{B-calc}
    B(t,s)x - B(t',s)x = \int_{t'}^t \partial_{\tau}B(\tau,s)x\, d\tau.
\end{equation}
This equation shows that $t\mapsto B(t,s)x$ is of class $C^1$, which is equivalent to the 
$C^1$-condition \eqref{C1-A}. Hence (iii) is equivalent to the condition \eqref{C1-A} and it remains to 
derive (i) and (ii) from (iii). By the strong continuity of $\tau\mapsto\partial_{\tau}B(\tau,s)$ and by the principle of uniform boundedness,
\begin{equation}\label{B-univ}
    \sup_{\tau\in I}\| \partial_{\tau}B(\tau,s)\| <\infty.
\end{equation}
Combining \eqref{B-calc} with \eqref{B-univ}, we see that $B(t,s)$ is
of bounded variation as a function of $t$, which is statement (ii), and that
$t\mapsto B(t,s)$ is continuous in norm. Therefore the inverse $t\mapsto B(t,s)^{-1}=B(s,t)$ is continuous as well and  $B(t,s) = B(t,0)B(0,s)$
is uniformly bounded for $t,s\in I$.
\end{proof}

The following Assumption~\ref{ass:Yosida} collects the regularity conditions from 
Yosida's Theorem~XIV.4.1,~\cite{YosidaFA6}.

\begin{assume}[Yosida] \label{ass:Yosida}\text{}
\begin{itemize}
\item[(i)] 
$\{ (s',t') \in I^2: s' \ne t' \} \ni (s,t) \mapsto \frac{1}{t-s} \, C(t,s)x$ is bounded and uniformly continuous for all $x \in X$, where $C(t,s) := A(t) A(s)^{-1} - 1$ 
\item[(ii)] 
$C(t)x := \lim_{k \to \infty} k \, C(t, t-\frac{1}{k})x$ exists uniformly in $t \in (0,1]$ for all $x \in X$
\item[(iii)] $(0,1] \ni t \mapsto C(t)x$ is continuous for all $x \in X$. 
\end{itemize}
\end{assume} 

The continuity assumption (iii) above was added for convenience. It
follows from the uniform continuity in (i) and the uniform convergence in
(ii). In fact, (i) and (ii) imply that $(0,1] \ni t \mapsto C(t)x$ is
\emph{uniformly} continuous and hence can be extended continuously to
the left end point $0$. 

The following theorem is our main result.
The key ingredient for its proof is the Lemma~\ref{lemma}, below.

\begin{theorem} \label{thm:main2}
Suppose that for each $t\in I$ the linear operator $A(t):D\subset X\to X$ is closed and that 
$A(t)$ has a bounded inverse. Then Assumption \ref{ass:Yosida}  and the $C^1$-condition \eqref{C1-A} are equivalent. 
\end{theorem}

\noindent\emph{Remark.} Note that the bounded invertibility is no restriction. If $A(t)$ for each $t\in I$ is the generator of a contraction semigroup,
then so is $A(t)-1$ and moreover, by Hille--Yosida, $A(t)-1$ has a bounded inverse.

\begin{proof}
Assumption~\ref{ass:Yosida} $\Rightarrow$ \eqref{C1-A}:
Suppose that conditions (i) - (iii) of Assumption~\ref{ass:Yosida} are satisfied and let
$x \in D$. We show that the map $t \mapsto f(t) = A(t)x$ satisfies the
hypotheses of Lemma~\ref{lemma}, below. By definition of $C(t,s)$, 
\begin{equation}\label{f-C}
   f(t) - f(s) = (A(t)A(s)^{-1} - 1)A(s)x = C(t,s)f(s)
\end{equation}
where, by~(i), the norm of $(t-s)^{-1}C(t,s)f(s)$ for fixed $s$ is a bounded
function of $t\in [0,1]\backslash\{s\}$. It follows that
$f(t)-f(s) \to 0$ as $t\to s$. As a further consequence of (i) we
obtain, by the principle of uniform boundedness, that
\begin{equation}\label{C-bound}
  M:= \sup_{s\neq t}\|(t-s)^{-1}C(t,s)\| <\infty.
\end{equation}
Setting $s=t-k^{-1}$ in \eqref{f-C} we obtain for fixed
$t>0$ and all $k>t^{-1}$ that 
\begin{align}\nonumber
    k \Bigl( f(t) - f\bigl(t-\frac{1}{k}\bigr)\Bigr)  =&\  kC\bigl(t,t-\frac{1}{k}\bigr) f\bigl(t-\frac{1}{k}\bigr)\\
   =&\  kC\bigl(t,t-\frac{1}{k}\bigr) f(t) +kC\bigl(t,t-\frac{1}{k}\bigr)\Bigl(f\bigl(t-\frac{1}{k}\bigr)-f(t)\Bigr)\label{df}\\ &\longrightarrow  C(t)f(t) \qquad (k \to \infty).\nonumber
\end{align}
Here we used part (ii) of Assumption~\ref{ass:Yosida},  the continuity
of $f$, and that $\|kC(t,t-k^{-1})\|\leq M$  by \eqref{C-bound}. Since $f$ is uniformly continuous on
the compact interval $I$ the second term of \eqref{df} vanishes uniformly in $t$. It remains to prove uniform convergence
for the first term of \eqref{df}. Using the uniform continuity of $f$
again, we may choose a partition $0=t_0<t_1\ldots <t_N=1$ of $[0,1]$ such that $\|f(t) -
f(t_i)\| \leq \eps/3M$ for all $t\in (t_{i-1},t_i]$ and all $i$. Then, by (ii), we can find $k_\eps$ such that for $k\geq k_\eps$
the inequality $\|kC\bigl(t, t-\frac{1}{k}\bigr)f(t_i)-C(t)f(t_i)\| <\eps/3$ holds for all $t\in [k^{-1},1]$ and
all $i=1,\ldots,N$. These two estimates combined imply that 
$$
    \sup_{t\in [k^{-1},1]}\bigg\| kC\bigl(t,t-\frac{1}{k}\bigr) f(t) - C(t)f(t)\bigg\| \leq \eps\qquad \text{for}\ k\geq k_\eps,
$$ 
as desired. Finally we note that the limit map $(0,1] \ni t \mapsto C(t)A(t)x$ is continuously extendable to the left endpoint $0$ by the remark following 
Assumption~\ref{ass:Yosida}. We have thus verified all hypotheses of Lemma~\ref{lemma}, and this lemma shows that the $C^1$-condition~\eqref{C1-A} is satisfied.

\medskip
\eqref{C1-A} $\Rightarrow$ Assumption~\ref{ass:Yosida}: Suppose  that
\eqref{C1-A} is satisfied and let $\dot{A}(t)x$ denote the derivative
of $A(t)x$. Then $s\mapsto A(s) A(0)^{-1}$ is strongly continuously differentiable and
hence continuous in norm. It follows that the inverse  
$s \mapsto \bigl( A(s) A(0)^{-1} \bigr)^{-1} = A(0) A(s)^{-1}$ is
norm-continuous as well. Thus, by \eqref{C1-A}, the map
\begin{align*}
  (s,\tau) \mapsto \dot{A}(\tau) A(s)^{-1} x = \dot{A}(\tau)A(0)^{-1} \, A(0)A(s)^{-1} x
\end{align*}
is continuous for every $x \in X$. From this, using the integral representation 
\begin{align*}
\frac{1}{t-s} \, C(t,s)x = \frac{1}{t-s} \, \bigl( A(t) - A(s) \bigr) A(s)^{-1} x = \frac{1}{t-s} \, \int_s^t \dot{A}(\tau) A(s)^{-1} x \, d\tau,
\end{align*} 
one readily obtains that $\{ s' \ne t' \} \ni (s,t) \mapsto  \frac{1}{t-s} \, C(t,s)x$ extends to a continuous map on the whole of $I^2$ from which conditions~(i) through~(iii) of Assumption~\ref{ass:Yosida} are obvious. 
\end{proof}

The exposition of Yosida's proof given in \cite{Schmid2011} shows that  the continuity in part (i) of Assumption~\ref{ass:Yosida} may be dropped
if in part (iii) the requirement is added that the limit $\lim_{t\searrow 0}C(t)x$ exists for all $x\in X$. Our proof of Theorem~\ref{thm:main2} shows, that this
modified version of  Assumption~\ref{ass:Yosida} is still equivalent to the $C^1$-condition \eqref{C1-A}.


The main ingredient for the proof of Theorem~\ref{thm:main2} is the following lemma.
It is a discretized version of the well-known, elementary
fact that a continuous and left-differentiable map with vanishing left
derivative is constant (see Lemma~III.1.36 in \cite{Kato-book} or
Corollary~1.2, Chapter 2 of \cite{Pazy}).

\begin{lemma} \label{lemma}
Suppose $f: [0,1] \to X$ is continuous and the limit $g(t) := \lim_{k \to \infty} k \, \bigl( f(t) - f(t-\frac{1}{k}) \bigr)$ exists uniformly in $t \in (0,1]$, that is, the limit exists for every $t \in (0,1]$ and $\sup_{t \in [\frac{1}{k},1]} \norm{ k \, \bigl( f(t) - f(t-\frac{1}{k}) \bigr) - g(t)} \longrightarrow 0 \,\,(k \to \infty)$. Then 
\begin{equation}\label{mve}
     \norm{ f(1) - f(t) } \le (1-t) \sup_{\tau \in [t,1]}\norm{g(\tau)} \quad\text{for all}\ t\in (0,1],
\end{equation}
and $f$ is continuously differentiable in $(0,1]$ with $f'=g$. If, in addition, the limit $g(0):=\lim_{t \searrow 0} g(t)$ exists, then $f'=g$ on $[0,1]$.
\end{lemma} 
 
\begin{proof}
The map $g$ is continuous on $(0,1]$ by the continuity of $f$ and the uniform convergence assumption. By the density of $(0,1) \cap \Q$  in $I$ and the continuity of
$f$ and $g$ on $(0,1]$ it suffices to show that, for every $\eps>0$ and for every $q\in
(0,1) \cap \Q$, the estimate
\begin{align} \label{zwbeh}
\norm{ f(1) - f(q) } \le (1-q)(M_q + \eps)
\end{align}
holds with $M_q:= \sup_{\tau \in [q,1]}\norm{g(\tau)}<\infty$. Let $q=1-r/s$
with $r,s\in \N$ and let $\eps >0$. For any $n\in \N$ we may write 
the difference $f(1) - f(1-r/s)$ as a telescoping sum
\begin{align} \label{teleskopsumme}
f(1) - f\Big(1-\frac{r}{s} \Big) = f(1) - f\Big(1- \frac{n r}{n s} \Big) = \sum_{k=0}^{n r-1} \left[f\Big(1-\frac{k}{n s} \Big) - f\Big(1-\frac{k}{n s}-\frac{1}{n s} \Big)\right]
\end{align}
where, by the assumed uniform convergence, we may choose $n$ so large, that
\begin{align} \label{estimate}
\sup_{t \in [q, 1] } \norm{ f(t) - f\Big(t-\frac{1}{ns} \Big) } \le (M_q + \eps) \, \frac{1}{n s}.
\end{align}
Combining~\eqref{teleskopsumme} and~\eqref{estimate} we immediately
obtain~\eqref{zwbeh} and the proof of the estimate~\eqref{mve} is complete. 

\medskip If the limit $\lim_{t \searrow 0} g(t)$ exists,
we can define $h(t) := f(t) - \int_{0}^t g(\tau) \,d\tau$ for $t \in
I$. It is straightforward to check that  $ k(h(t) - h(t-k^{-1}) \to
0$ uniformly and the established estimate \eqref{mve} yields the constancy of $h$. The
proof of the remaining statement of the lemma is an easy exercise that
is left to the reader.
\end{proof}


\end{document}